\documentclass[reqno]{amsart}
\usepackage{amsfonts,amsmath,amsthm,amssymb,mathrsfs}
\usepackage{color}
\usepackage{amsmath,amssymb,amsfonts,bm}
\usepackage{graphicx}

\usepackage{CJKutf8}
\usepackage{xcolor}
\numberwithin{equation}{section}
\usepackage[pdfstartview=FitH,colorlinks=true]{hyperref}
\hypersetup{urlcolor=red, linkcolor=blue,citecolor=red, CJKbookmarks=true}
\allowdisplaybreaks

\newtheorem{thm}{Theorem}[section]

\newtheorem{lem}[thm]{Lemma}
\newtheorem{remark}[thm]{Remark}
\newtheorem{prop}[thm]{Proposition}
\newtheorem{defn}[thm]{Definition}

\theoremstyle{plain}
\newtheorem{rem}[thm]{Remark}

\newcommand{\eps}{\varepsilon}

\newcommand{\NN}{\mathbb{N}}

\newcommand{\RR}{\mathbb{R}}

\newcommand{\dd}{\partial}

\newcommand{\Ccal}{\mathcal{C}}

\newcommand{\Pcal}{\mathcal{P}}

\newcommand{\abs}[1]{\left\vert#1\right\vert}

\newcommand{\norm}[1]{\left\Vert#1\right\Vert}

\newcommand{\comii}[1]{\left<#1\right>}

\usepackage{empheq}

\definecolor{myyellow}{RGB}{255,255,190}
\definecolor{myrose}{RGB}{255,218,185}
\definecolor{mygreen}{RGB}{152,251,152}

\newlength\mytemplen
\newsavebox\mytempbox

\makeatletter

\newcommand\yellowbox{%
    \@ifnextchar[
       {\@yellowbox}%
       {\@yellowbox[0pt]}}

\def\@yellowbox[#1]{%
    \@ifnextchar[
       {\@@yellowbox[#1]}%
       {\@@yellowbox[#1][0pt]}}

\def\@@yellowbox[#1][#2]#3{
    \sbox\mytempbox{#3}%
    \mytemplen\ht\mytempbox
    \advance\mytemplen #1\relax
    \ht\mytempbox\mytemplen
    \mytemplen\dp\mytempbox
    \advance\mytemplen #2\relax
    \dp\mytempbox\mytemplen
    \colorbox{myyellow}{\hspace{1em}\usebox{\mytempbox}\hspace{1em}}}

\newcommand\greenbox{%
    \@ifnextchar[
       {\@greenbox}%
       {\@greenbox[0pt]}}

\def\@greenbox[#1]{%
    \@ifnextchar[
       {\@@greenbox[#1]}%
       {\@@greenbox[#1][0pt]}}

\def\@@greenbox[#1][#2]#3{
    \sbox\mytempbox{#3}%
    \mytemplen\ht\mytempbox
    \advance\mytemplen #1\relax
    \ht\mytempbox\mytemplen
    \mytemplen\dp\mytempbox
    \advance\mytemplen #2\relax
    \dp\mytempbox\mytemplen
    \colorbox{mygreen}{\hspace{1em}\usebox{\mytempbox}\hspace{1em}}}

\newcommand\rosebox{%
    \@ifnextchar[
       {\@rosebox}%
       {\@rosebox[0pt]}}

\def\@rosebox[#1]{%
    \@ifnextchar[
       {\@@rosebox[#1]}%
       {\@@rosebox[#1][0pt]}}

\def\@@rosebox[#1][#2]#3{
    \sbox\mytempbox{#3}%
    \mytemplen\ht\mytempbox
    \advance\mytemplen #1\relax
    \ht\mytempbox\mytemplen
    \mytemplen\dp\mytempbox
    \advance\mytemplen #2\relax
    \dp\mytempbox\mytemplen
    \colorbox{myrose}{\hspace{1em}\usebox{\mytempbox}\hspace{1em}}}

\makeatother

\begin{document}

\title[Gevrey regularity with weight for  Euler equation ]
{Gevrey regularity with weight for incompressible Euler equation in the half plane}

\author[F. Cheng \and W.-X. Li \and C.-J. Xu]
{Feng Cheng \and Wei-Xi Li \and Chao-Jiang Xu}

\date{}

\address{ Feng Cheng,  School of Mathematics and Statistics, Wuhan University,  430072 Wuhan,  China}
\email{chengfengwhu@whu.edu.cn}

\address{ Wei-Xi Li,  School of Mathematics and Statistics, and Computational Science Hubei Key Laboratory,   Wuhan University,  430072 Wuhan, China}
\email{wei-xi.li@whu.edu.cn}

\address{ Chao-Jiang Xu,
School of Mathematics and Statistics,  Wuhan University,  430072 Wuhan, China}
\email{chjxu.math@whu.edu.cn}


\begin{abstract}
    In this work we prove the weighted Gevrey regularity of solutions to the   incompressible Euler equation  with initial data decaying polynomially at infinity. This   is  motivated by  the   well-posedness problem of vertical boundary layer equation  for fast rotating fluid.      The method  presented  here is based on the basic weighted $L^2$- estimate,  and the main difficulty arises from  the estimate on the pressure term due to the appearance of  weight function.
\end{abstract}

\maketitle

\section{Introduction}

In this paper we study the Gevrey  propagation of solutions to incompressible Euler equation.  Gevrey  class   is a stronger concept than the $C^\infty$-smoothness. In fact it is an intermediate space between analytic space and $C^\infty$ space.
There have been extensive mathematical investigations (cf.\cite{BoB}, \cite{EM}, \cite{KV2}, \cite{KV3} \cite{Ka}, \cite{Ka2}, \cite{T}, \cite{Foias}, \cite{Y} for instance and the references therein) on Euler equation in different kind of frames,  such as  Sobolev space, analytic space and Gevrey space.    In this work we will consider the problems of Gevrey regularity with weight, and this is motivated by the study of the vertical boundary layer problem.   From a physical point of view, as well as from a mathematical point of view,  when the direction of rotation is  perpendicular  to
the boundaries,  the boundary equation is well developed and called Ekman layers.  Up to now  the Ekman layers (horizontal layer) are  well understood, (c.f.  \cite{cdgg} for instance and the references therein).
 When the direction of rotation is  parallel  to
the boundaries,  the situation is, however, different from the perpendicular case above:   the vertical layers are very different and much more complicated, from a physical, analytical and mathematical point of view, and many open questions in all these directions remain open.  We refer to  \cite{cdgg} for detailed discussion on the vertical layers.   Recently we try to investigate the well-posedness problem for vertical layer,  and the related and preliminary step  is to establish the Gevrey regularty with weight for the outer flow which is described by Euler equation. This is the main result of the present paper.

Without loss of generality, we consider the incompressible Euler equation in half plane $\RR^2_+$, where $\RR^2_+= \{(x,y);x\in\RR, y\in\RR^+ \}$, and our results can be generalized to 3-D Euler equation. The velocity $(u(t,x,y),v(t,x,y))$ and the pressure $p(t,x,y)$ satisfy the following equation:
\begin{align}
        \dd_t u + u\dd_x u+v\dd_y u+ \dd_x p &=0\quad \ \mbox{in}\ (0,\infty)\times\RR^2_+\label{1.1}\\      \dd_t v + u\dd_x v+v\dd_y v+\dd_y p  &=0\quad \ \mbox{in}\ (0,\infty)\times\RR^2_+\label{1.2}\\
        \partial_x u+\partial_y v &=0\quad \ \mbox{in}\ (0,\infty)\times\RR^2_+\label{1.3}
\end{align}
The boundary condition
\begin{equation}\label{1.4}
        v\big|_{y=0} =0 \quad \ \text{in}\ (0,\infty)\times\RR;\quad u,v\to0,\ \text{as}\  \sqrt{x^2+y^2}\to\infty
\end{equation}
With initial data
\begin{equation}\label{1.5}
        u|_{t=0}=u_0, \quad v|_{t=0}=v_0,\quad\ \mbox{in}\ \RR^2_+.
\end{equation}
Here the initial data $(u_0,v_0)$ satisfy the compatibility condition:
\[
   \dd_x u_0+\dd_y v_0=0,\quad v_0|_{y=0}=0,\quad u_0,v_0\to0,\quad\text{as}\quad \sqrt{x^2+y^2}\to\infty.
\]

Before stating our main result we first introduce the (global) weighted Gevrey space.

\begin{defn}\label{defgev}
Let $\ell_x,\ell_y\geq 0$ be real constants that independent of $x,y$, we say that $f\in G^s_{\tau,\ell_x,\ell_y}(\RR^2_+)$ if
\begin{equation*}
      \sup_{\abs\alpha\geq 0}  \frac{\tau^\alpha}{\abs\alpha!^s} \norm{\comii{x}^{\ell_x}\comii{y}^{\ell_y} \dd^\alpha f}_{L^2(\RR^2_+)}<\infty,
\end{equation*}
where and throughout the paper we use the notation  $\comii{\cdot} =(1+|\cdot|^2)^{\frac{1}{2}}$.
\end{defn}

In this  work we present the persistence of weighted Gevrey class regularity of the solution, i.e., we prove that if  the initial datum $(u_0,v_0)$ is in some weighted Gevrey   space and satisfy the compatible condition, then the global  solution belongs to the same space.   With only minor changes, these results can also be extended to 3-D Euler equation, and the global solution here will be replaced by a local solution. 	
Precisely,
\begin{thm}\label{thm1}
Suppose the initial data $u_0\in G^s_{\tau_0,0,\ell_y}, v_0\in G^s_{\tau_0,\ell_x,0}$ for some $s\geq 1,\tau_0>0$ and $0\leq\ell_x,\ell_y\leq 1$.  Then the Euler equation (\ref{1.1})-(\ref{1.5}) admits a solution $u,v,p$ such that
\begin{eqnarray*}
u(t,\cdot)\in L^\infty\left([0,+\infty);~G^s_{\tau(t),0,\ell_y}\right), \quad  v(t,\cdot)\in L^\infty\left([0,+\infty);~G^s_{\tau(t),\ell_x,0}\right)
\end{eqnarray*}
and p satisfies
\[
\dd_x p(t,\cdot)\in L^\infty\left([0,+\infty);\ G^s_{\tau(t),0,\ell_y}\right), \quad \dd_y p(t,\cdot)\in L^\infty\left([0,+\infty);\ G^s_{\tau(t),\ell_x,0}\right)
\]
where $\tau(t)$ is a decreasing function of $t$ with initial value $\tau_0$.
\end{thm}
\begin{rem}
If the initial data $(u_0,v_0)$ were posed both horizontal and vertical weights, namely $u_0\in G^s_{\tau_0,\ell_x,\ell_y},v_0\in G^s_{\tau_0,\ell_x,\ell_y}$,  the result of Theorem \ref{thm1} is also valid.
\end{rem}

We remark that the existence of smooth solutions to (\ref{1.1})-(\ref{1.5}) is well developed (c.f.\cite{BoB}, \cite{EM}, \cite{Ka}, \cite{T}, \cite{Y} for instance), and in two-dimensional case smooth initial data can yield global solutions, while in the three-dimensional case the solution may be local in general condition. The appearance of the weight function increases the difficulty of estimating the pressure term, and for this part it is different from   \cite{KV2}.    We also point out that  in the whole space $\mathbb R^2$ or two dimensional torus $\mathbb T^2,$  the classical  approach  to  analyticity or Gevrey regularity is that it   makes crucial use of   Fourier transformation, which can't apply to our case. Instead we will use the basic $L^2$ estimate (c.f. \cite{clx1, clx2,clx3, clx4,KV2} for instance).

The paper is organized as follows.  In section \ref{sec2}, we introduce the notation used to define the weighted Sobolev norms, and we prove the persistence of the weighted Sobolev regularity. In section \ref{sec3}, we state   the priori estimate to prove the main theorem. Section \ref{sec4} and \ref{sec5} are consist of the proofs of these lemmas.

\section{Notations and Preliminaries}\label{sec2}
In the following context, we use the conventional symbols for the standard Sobolev spaces \({ H}^m(\RR^2_+) \) with \(m\in \mathbb{N}\),  and let $\norm{\cdot}_{H^m}$ be its norm. For the case \(m=0\), it was  usually written as ${ L}^2(\RR^2_+)$. Denote $\left\|\cdot\right\|_{L^2}$ and $\left<\cdot,\cdot\right>$ be the norm and inner product in ${ L}^2(\RR^2_+)$. We usually write a vector function in bold type as ${\bf u}$ and a scalar function in it's conventional way as $u$. For a vector function ${\bf u}=(u,v)$ we denote
  \[
   \left\|{\bf u} \right\|_{H^m}=\sqrt{\norm{u}^2_{H^m}+\norm{v}^2_{H^m}}.
  \]
And when we say that ${\bf u}\in { H}^m$, we mean that $u,v\in{ H}^m$.

With the notations above, we introduce the weighted Sobolev spaces ${ H}^m_{\ell_x}(\RR^2_+)$ and ${ H}^m_{\ell_y}(\RR^2_+)$, where $\ell_x,\ell_y$ are real constants. Let
$$
{H}^m_{\ell_x}(\RR^2_+)=\left\{v\in{ H}^m(\RR^2_+);\quad \comii{x}^{\ell_x}\dd^\alpha v\in L^2, 1\leq\abs{\alpha}\leq m   \right\},
$$
and it's norm is defined by
 \[
  \norm{v}_{H^m_{\ell_x}}=\sqrt{\norm{v}_{L^2}^2+\sum_{1\leq\abs{\alpha}\leq m}\norm{\comii{x}^{\ell_x}\dd^\alpha v}_{L^2}^2 }.
 \]
 Similarly, let
$$
{ H}^m_{\ell_y}(\RR^2_+)=\left\{u\in{ H}^m(\RR^2_+);\quad \comii{y}^{\ell_y}\dd^\alpha u\in L^2, 1\leq\abs{\alpha}\leq m    \right\},
$$
equipped with the norm
 \[
  \norm{u}_{H^m_{\ell_y}}=\sqrt{\norm{u}_{L^2}^2+\sum_{1\leq\abs{\alpha}\leq m}\norm{\comii{y}^{\ell_y}\dd^\alpha u}_{L^2}^2 }.
 \]
We then define  space $H^m_{\ell_x,\ell_y}(\RR^2_+)$ of vector functions by
 \[
   H^m_{\ell_x,\ell_y}(\RR^2_+)=\left\{{\bf u}=(u,v)\in H^m_{\ell_x,\ell_y}(\RR^2_+): u\in H^m_{\ell_y}(\RR^2_+), v\in H^m_{\ell_x}(\RR^2_+)\right\}
 \]
which is equipped with the norm
 \[
  \norm{{\bf u}}_{H^m_{\ell_x,\ell_y}}=\sqrt{\norm{u}_{H^m_{\ell_y}}^2+\norm{v}_{H^m_{\ell_x}}^2 }.
 \]

It's well known that the corresponding Cauchy problem to \eqref{1.1}-\eqref{1.5} is globally well posed in $H^m$ if $m>2$ with dimension $d=2$, see e.g.[\cite{MT}, Chapter 17, Section 2] and
\begin{equation*}
     \norm{\bf u(t)}_{{ H}^m}\leq \norm{{\bf u}_0}_{{ H}^m}\exp\left(C_0\int_0^t \norm{\nabla {\bf u(s)}}_{{ L}^\infty}ds \right)
\end{equation*}
Where $C_0$ is a constant depending on ${\bf u}_0$.

Now we  will show that if the initial data ${\bf u}_0=(u_0,v_0)\in H^m_{\ell_x,\ell_y}$ for $0\leq \ell_x,\ell_y\leq1$, the solution is also in $H^m_{\ell_x,\ell_y}$. This is the first step for the weighted Gevery regularity.

\begin{prop}\label{prop1}
For fixed $m\geq3$ and $0\leq\ell_x,\ell_y\leq1$, let the initial data   ${\bf u}_0\in {H}^m_{\ell_x,\ell_y}(\RR^2_+)$ and  suppose the compatibility condition is fulfilled.
Then the  ${ H}^m$-solution ${\bf u}$ to the Euler equation (\ref{1.1})-(\ref{1.5}) is also in weighted Sobolev space:
$$
{\bf u}(t,\cdot)\in L^\infty\left([0,\infty);\ { H}^m_{\ell_x,\ell_y}(\RR^2_+)\right).
$$
Moreover
\begin{equation}\label{estq}
\begin{aligned}
 \norm{{\bf u}(t,\cdot)}_{H^m_{\ell_x,\ell_y}} \leq \norm{{\bf u}_0}_{H^m_{\ell_x,\ell_y}}\exp &\bigg[C_0 \int_0^t \bigg(\norm{{\bf u}(s,\cdot)}_{{ L}^\infty}+\norm{\comii{y}^{\ell_y}\nabla u(s,\cdot)}_{{ L}^\infty}\\
 &+\norm{\comii{x}^{\ell_x}\nabla v(s,\cdot)}_{{ L}^\infty} \bigg)ds\bigg],
\end{aligned}
\end{equation}
where $C_0$ is a constant depending on m.
\end{prop}

\begin{proof}
It suffices to show (\ref{estq}) holds. When no ambiguity arises, we suppress the time dependence of $u$ and $v$ on $t$.  We begin with proving a priori estimate.    First we have
 \begin{equation}\label{2.2}
{1\over2}\frac{d}{dt} \left(\norm{u(t,\cdot)}_{L^2}^2+\norm{v(t,\cdot)}_{L^2}^2 \right)=0.
\end{equation}
Now let $\alpha\in \mathbb{N}_0^2$ be the multi-index such that $1\leq |\alpha|\leq m$. We apply $\partial^\alpha$ on both sides of \eqref{1.1} and take $L^2-$ inner product with $\comii{y}^{2\ell_y}\dd^\alpha u$
\begin{equation}\label{2.3}
{1\over2}\frac{d}{dt} \norm{\comii{y}^{\ell_y}\dd^\alpha u }_{L^2}^2+ \left< \left<y\right>^{\ell_y}\partial^\alpha({\bf u}\cdot\nabla u),
\left<y\right>^{\ell_y}\partial^\alpha u\right>+\left<\left<y\right>^{\ell_y}\partial^\alpha\partial_x p ,\left<y\right>^{\ell_y}\partial^\alpha u \right>=0.
\end{equation}
And similarly taking $L^2-$ inner product with $\comii{x}^{2\ell_x}\dd^\alpha v$ on both sides of \eqref{1.2},
\begin{equation}\label{2.4}
{1\over2}\frac{d}{dt}\norm{\comii{x}^{\ell_x}\dd^\alpha v}_{L^2}^2+\left<\comii{x}^{\ell_x}\dd^\alpha({\bf u}\cdot\nabla v),\comii{x}^{\ell_x}\dd^\alpha v\right>+
\left<\comii{x}^{\ell_x}\dd^\alpha\dd_y p,\comii{x}^{\ell_x}\dd^\alpha v\right>=0.
\end{equation}
Taking sum over $1\leq\abs{\alpha}\leq m$ in (\ref{2.3}) and (\ref{2.4}), and combining (\ref{2.2}), we have
\begin{equation}\label{2.5}
\begin{aligned}
 &\quad{\frac{1}{2}}\frac{d}{dt}\norm{{\bf u}(t,\cdot)}_{H^m_{\ell_x,\ell_y}}^2\\
 &+\sum_{1\leq\abs{\alpha}\leq m} \left[ \left< \left<y\right>^{\ell_y}\dd^\alpha({\bf u}\cdot\nabla u),
\left<y\right>^{\ell_y}\partial^\alpha u\right> + \left<\comii{x}^{\ell_x}\dd^\alpha({\bf u}\cdot\nabla v),\comii{x}^{\ell_x}\dd^\alpha v\right> \right] \\
 &+\sum_{1\leq\abs{\alpha}\leq m} \left[\left<\left<y\right>^{\ell_y}\partial^\alpha\partial_x p ,\left<y\right>^{\ell_y}\partial^\alpha u \right>
 +\left<\comii{x}^{\ell_x}\dd^\alpha\dd_y p,\comii{x}^{\ell_x}\dd^\alpha v\right>\right]=0.
\end{aligned}
\end{equation}
It remains to estimate ${\rm I}_1$ and ${\rm I}_2$, with $I_j$ defined by
\begin{eqnarray*}
{\rm I}_1&=&\sum_{1\leq\abs{\alpha}\leq m} \left[ \left< \left<y\right>^{\ell_y}\dd^\alpha({\bf u}\cdot\nabla u),
\left<y\right>^{\ell_y}\partial^\alpha u\right> +
 \left<\comii{x}^{\ell_x}\dd^\alpha({\bf u}\cdot\nabla v),\comii{x}^{\ell_x}\dd^\alpha v\right> \right],\\
{\rm I}_2&=&\sum_{1\leq\abs{\alpha}\leq m} \left[\left<\left<y\right>^{\ell_y}\partial^\alpha\partial_x p ,\left<y\right>^{\ell_y}\partial^\alpha u \right>
+\left<\comii{x}^{\ell_x}\dd^\alpha\dd_y p,\comii{x}^{\ell_x}\dd^\alpha v\right>\right].
\end{eqnarray*}

{\bf The estimate on}  {\bm  ${\rm I}_1$}:

Using H\"older inequality and divergence-free condition, we have
\begin{equation*}
\begin{aligned}
\abs{{\rm I}}_1 &\leq \norm{{\bf u}}_{H^m_{\ell_x,\ell_y}}\sum_{1\leq\abs{\alpha}\leq m}\left[\norm{\comii{y}^{\ell_y}\dd^\alpha({\bf u}\cdot\nabla u)-{\bf u}\cdot\nabla(\comii{y}^{\ell_y}\dd^\alpha u)}_{L^2}\right.\\
&\left.\quad+\norm{\comii{x}^{\ell_x}\dd^\alpha({\bf u}\cdot\nabla v)-{\bf u}\cdot\nabla(\comii{x}^{\ell_x}\dd^\alpha v)}_{L^2} \right].
\end{aligned}
\end{equation*}
Note the fact that $\left|\dd^\beta\comii{y}^{\ell_y}\right|,\left|\dd^\beta\comii{x}^{\ell_x}\right|\leq C_m$ for $1\leq\beta\leq\alpha$ and $C_m$ depending on m, then the weight function can be put into the bracket. And with the application of \cite[(3.22),Chapter 13, Section 3]{MT} we have
\begin{equation}\label{2.6}
 \abs{{\rm I}_1}\leq C\left(\norm{\bf u}_{L^\infty}+\norm{\comii{y}^{\ell_y}\nabla u}_{L^\infty}+\norm{\comii{x}^{\ell_x}\nabla v}_{L^\infty} \right)\norm{\bf u}_{H^m_{\ell_x,\ell_y}}^2.
\end{equation}
Where $C$ is a constant depending on $m$. In the following $C$ denotes a generic positive constant depending on $m$.

{\bf The estimate on} {\bm  ${\rm I_2}$:}

In order to estimate ${\rm I}_2$, we need to use Lemma \ref{lemma3} and Lemma \ref{lemma4} in Section \ref{sec5}.    Observe $p$ satisfies the following Neumann problem.
\begin{equation*}
\left\{
\begin{aligned}
 \Delta p     &=2(\dd_x u)\dd_y v-2(\dd_y u)\dd_x v \quad \text{in}\ \RR^2_+\times\{0,\infty \},\\
 \dd_y p|_{y=0}&=0 \quad \mbox{on}\ \RR\times\{0,\infty \}.
\end{aligned}
\right.
\end{equation*}
We proceed to estimate ${\rm I}_2$ through two cases.

(a). If $\abs{\alpha}=1$, then we use Lemma \ref{lemma3} and classical argument of $H^2$-regularity result of the above Neumann problem, to get
  \begin{equation*}
  \begin{aligned}
  &\quad\sum_{\abs{\alpha}=1}\left( \norm{\comii{y}^{\ell_y}\dd_x\dd^\alpha p}_{L^2}+\norm{\comii{x}^{\ell_x}\dd_y\dd^\alpha p}_{L^2}\right)\\
  &\leq C\norm{\comii{y}^{\ell_y}\dd_y u\dd_x v}_{L^2}+C\norm{\comii{y}^{\ell_y}\dd_x u\dd_y v}_{L^2}+C\norm{\comii{x}^{\ell_x}\dd_y u\dd_x v}_{L^2}\\
  &\quad+C\norm{\comii{x}^{\ell_x}\dd_x u\dd_y v}_{L^2}+C\norm{\dd_x p}_{L^2}+C\norm{\dd_y p}_{L^2}\\
  &\leq C\norm{\bf u}_{{\rm H}^m_{\ell_x,\ell_y}}\norm{\nabla{\bf u}}_{L^\infty},
  \end{aligned}
  \end{equation*}
where we used the Hodge decomposition of ${ L}^2(\RR^2_+)$ to estimate $\norm{\nabla p}_{L^2}$ and C is a constant.

(b). If $2\leq\abs{\alpha}=k\leq m$, then we use Lemma \ref{lemma4} and similar arguments as \cite[Proposition 3.6, Chapter 13, Section 3]{MT}; this gives
  \begin{equation*}
  \begin{aligned}
  &\quad\sum_{2\leq\abs{\alpha}\leq m}\bigg( \norm{\comii{y}^{\ell_y}\dd_x\dd^\alpha p}_{L^2}+\norm{\comii{x}^{\ell_x}\dd_y\dd^\alpha p}_{L^2}\bigg)\\
  &\leq C\sum_{k=2}^m\sum_{\abs{\beta}=k-1}\bigg(\norm{\comii{y}^{\ell_y}\dd^\beta(\dd_y u\dd_x v)}_{L^2}+\norm{\comii{y}^{\ell_y}\dd^\beta(\dd_x u\dd_y v)}_{L^2} \\
  &\quad+\norm{\comii{x}^{\ell_x}\dd^\beta(\dd_y u\dd_x v)}_{L^2}
  +\norm{\comii{x}^{\ell_x}\dd^\beta(\dd_x u\dd_y v)}_{L^2}\bigg)\\
  &\quad+C\sum_{k=2}^{m}\bigg(\norm{\dd_x^{k-2}(\dd_y u\dd_x v)}_{L^2}+\norm{\dd_x^{k-2}(\dd_x u\dd_y v)}_{L^2}\bigg)\\
  &\leq C\norm{\bf u}_{{\rm H}^m_{\ell_x,\ell_y}}\bigg(\norm{\comii{y}^{\ell_y}\nabla u}_{L^\infty}+\norm{\comii{x}^{\ell_x}\nabla v}_{L^\infty}\bigg).
  \end{aligned}
  \end{equation*}

Thus we combine the above two cases to conclude that
\begin{equation}\label{2.7}
\abs{{\rm I}_2}\leq C\norm{\bf u}_{{ H}^m_{\ell_x,\ell_y}}^2\left( \norm{\bf u}_{L^\infty}+\norm{\comii{y}^{\ell_y}\nabla u}_{L^\infty}+\norm{\comii{x}^{\ell_x}\nabla v}_{L^\infty}\right),
\end{equation}
where $C$ is a constant depending only on $m$.

And then by (\ref{2.5}), (\ref{2.6}) and (\ref{2.7}), we have
\begin{equation*}
\frac{d}{dt}\norm{{\bf u}(t)}_{{ H}^m_{\ell_x,\ell_y}}\leq C\left(\norm{\bf u}_{L^\infty}+\norm{\comii{y}^{\ell_y}\nabla u}_{L^\infty}+\norm{\comii{x}^{\ell_x}\nabla v}_{L^\infty}\right)\norm{{\bf u}(t)}_{H^m_{\ell_x,\ell_y}}
\end{equation*}
Then with Grownwall inequality we obtain (\ref{estq}).

Now we consider  ${\bf u}\in H^m$.   Repeating the above arguments with $\comii y^{\ell_y}$   and $\comii x^{\ell_x}$ replaced,  respectively, by
\begin{eqnarray*}
	\frac{\comii y^{\ell_y} }{\comii {\eps y}^{\ell_y}}, \quad \frac{\comii x^{\ell_x} }{\comii {\eps x}^{\ell_x}}
\end{eqnarray*}
where $0<\eps<1$, then we can also deduce \eqref{estq} by letting $\eps\rightarrow 0$.  We then complete the proof of the proposition.
\end{proof}

\section{Weighted Gevrey regularity}\label{sec3}

We inherit the notations that used in \cite{KV2} for $X_\tau$ and $Y_\tau$. That is to say for a multi-index $\alpha=(\alpha_1,\alpha_2)$ in $\NN^2$, and a vector function ${\bf u}=(u,v)$,  define the Sobolev and semi-norms as follows:
   \[
    \abs{\bf u}_{m,\ell_x,\ell_y}=\sum_{\abs{\alpha}=m}\left(\norm{\comii{y}^{\ell_y}\dd^\alpha {u}}_{L^2}+\norm{\comii{x}^{\ell_x}\dd^\alpha v}_{L^2}\right),
   \]

   \[
    \abs{\bf u}_{m,\ell_x,\ell_y,\infty}=\sum_{\abs{\alpha}=m}\left(\norm{\comii{y}^{\ell_y}\dd^\alpha {u}}_{L^\infty}+\norm{\comii{x}^{\ell_x}\dd^\alpha v}_{L^\infty}\right),
   \]
   where $\abs{\bf u}_m=\abs{\bf u}_{m,0,0}$ and $\abs{\bf u}_{m,\infty}=\abs{\bf u}_{m,0,0,\infty}.$

For $s\geq1$ and $\tau>0$, define a new weighted Gevrey spaces, which is equivalent to that in Definition  \ref{defgev},  by
   \[
    X_{\tau,\ell_x,\ell_y}=\left\{{\bf u}\in C^\infty :\norm{\bf u}_{X_{\tau,\ell_x,\ell_y}}<\infty \right\},
   \]
where
   \[
    \norm{\bf u}_{X_{\tau,\ell_x,\ell_y}}=\sum_{m=3}^\infty \abs{\bf u}_{m,\ell_x,\ell_y}\frac{\tau^{m-3}}{(m-3)!^s}.
   \]
And let
   \[
    Y_{\tau,\ell_x,\ell_y}=\left\{{\bf u}\in C^\infty :\norm{\bf u}_{Y_{\tau,\ell_x,\ell_y}}<\infty \right\},
   \]
where
   \[
    \norm{\bf u}_{Y_{\tau,\ell_x,\ell_y}}=\sum_{m=3}^\infty \abs{\bf u}_{m,\ell_x,\ell_y}\frac{(m-3)\tau^{m-4}}{(m-3)!^s}.
   \]
We will denote  $X_{\tau}=X_{\tau,0,0}$ and $Y_\tau=Y_{\tau,0,0}$.

In order to show the main result, Theorem \ref{thm1}, it suffices to show the following
\begin{thm}\label{thsec}
Let   the initial data ${{\bf u}_0}=(u_0,v_0)$ satisfy
 \[{\bf u}_0\in X_{\tau_0,\ell_x,\ell_y}\] for some $s\geq 1,\tau_0>0$ and $0\leq\ell_x,\ell_y\leq 1$. Then the Euler system (\ref{1.1})-(\ref{1.5}) admits a solution
	\begin{eqnarray*}
		{\bf u}(t,\cdot)\in L^\infty\left([0,\infty);~X_{\tau(t),\ell_x,\ell_y}\right),
	\end{eqnarray*}
     where $\tau(t)$ is a decreasing function depending on the initial radius $\tau_0$ and the $H^m_{\ell_x,\ell_y}-$ solution ${\bf u}$ for $m\geq 6$.
\end{thm}

We will prove Theorem 3.1 using the method of \cite{KV2}, with main difference from the estimate on pressure.  By Proposition \ref{prop1} we see $\norm{\bf u}_{H^m_{\ell_x,\ell_y}}<+\infty$ for each $m>2$. When no ambiguity arises, we suppress the time dependence of $\tau$ and $u,v$ on $t$.  With notations above we have
    \begin{equation}\label{3.1}
      \frac{d}{dt}\norm{{\bf u}(t,\cdot)}_{X_{\tau(t),\ell_x,\ell_y}} =\tau^\prime(t)\norm{{\bf u}(t,\cdot)}_{Y_{\tau(t),\ell_x,\ell_y}}+\sum_{m=3}^\infty \frac{d}{dt}|{\bf u}(t,\cdot)|_{m,\ell_x,\ell_y} \frac{\tau(t)^{m-3}}{(m-3)!^s}.
    \end{equation}
Recalling from (\ref{2.4}) and (\ref{2.5}) and using H\"older inequality,  we obtain
\begin{equation*}
\begin{aligned}
 \frac{d}{dt}\abs{{\bf u}(t,\cdot)}_{m,\ell_x,\ell_y} &\leq \sum_{\abs{\alpha}=m}\sum_{\beta\leq\alpha,\beta\neq0}{\alpha\choose\beta}
 \bigg(\norm{\comii{y}^{\ell_y}\dd^\beta{\bf u}\cdot\nabla\dd^{\alpha-\beta}u}_{L^2} \\
 &\quad+\norm{\comii{x}^{\ell_x}\dd^\beta{\bf u}\cdot\nabla\dd^{\alpha-\beta}v}_{L^2} \bigg)
 +\sum_{\abs{\alpha}=m}\bigg(\norm{\comii{y}^{\ell_y}\dd_x\dd^\alpha p}_{L^2} \\
 &\quad+\norm{\comii{x}^{\ell_x}\dd_y\dd^\alpha p}_{L^2} \bigg)+\norm{\bf u}_{L^\infty}\abs{\bf u}_m.
\end{aligned}
\end{equation*}
Set
\begin{equation*}
\begin{aligned}
 \Ccal&=\sum_{m=3}^\infty \sum_{\abs{\alpha}=m} \sum_{\beta\leq\alpha,\beta\neq0}{\alpha\choose\beta}\bigg(\norm{\comii{y}^{\ell_y}\dd^\beta{\bf u}\cdot\nabla\dd^{\alpha-\beta}u}_{L^2}\\
 &\quad+\norm{\comii{x}^{\ell_x}\dd^\beta{\bf u}\cdot\nabla\dd^{\alpha-\beta}v}_{L^2} \bigg)\frac{\tau(t)^{m-3}}{(m-3)!^s}
\end{aligned}
\end{equation*}
and
\[
 \Pcal=\sum_{m=3}^\infty\sum_{\abs{\alpha}=m} \left(\norm{\comii{y}^{\ell_y}\dd_x\dd^\alpha p}_{L^2}+\norm{\comii{x}^{\ell_x}\dd_y\dd^\alpha p}_{L^2} \right)\frac{\tau(t)^{m-3}}{(m-3)!^s}.
\]
Combined with (\ref{3.1}), we have
    \begin{equation}\label{3.2}
      \frac{d}{dt}\norm{{\bf u}(t,\cdot)}_{X_{\tau(t),\ell_x,\ell_y}}\leq \tau^\prime(t)\norm{{\bf u}(t,\cdot)}_{Y_{\tau(t),\ell_x,\ell_y}}+\mathcal{C}+\mathcal{P}+C\tau(t)\norm{{\bf u}(t,\cdot)}_{L^\infty}\norm{{\bf u}(t,\cdot)}_{Y_\tau}.
    \end{equation}

We give the following Lemma to estimate $\Ccal$, the proof is postponed to Section \ref{sec4}.

    \begin{lem}\label{commutator}
    There exists a sufficiently large constant $C>0$ such that
      \[
       \Ccal\leq C\left(\Ccal_1+\Ccal_2\norm{\bf u}_{Y_{\tau,\ell_x,\ell_y}} \right),
      \]
    where
      \begin{equation*}
      \begin{aligned}
       \Ccal_1 &=\abs{\bf u}_{1,\ell_x,\ell_y,\infty}\abs{\bf u}_{3,\ell_x,\ell_y}+\abs{\bf u}_{2,\infty}\abs{\bf u}_{2,\ell_x,\ell_y}+\tau \abs{\bf u}_{2,\ell_x,\ell_y,\infty}\abs{u}_{3,\ell_x,\ell_y}\\
       &\quad+\tau^2\abs{\bf u}_3\abs{\bf u}_{3,\ell_x,\ell_y,\infty}
      \end{aligned}
      \end{equation*}
    and
      \[
       \Ccal_2=\tau\abs{\bf u}_{1,\ell_x,\ell_y,\infty}+\tau^2\abs{\bf u}_{2,\ell_x,\ell_y,\infty}+\tau^2\norm{\bf u}_{X_{\tau,\ell_x,\ell_y}}+\tau^3\abs{\bf u}_{3,\ell_x,\ell_y,\infty}.
      \]
    \end{lem}

The following lemmas shall be used to estimate $\Pcal$. The proof is postponed to Section \ref{sec5} below.

    \begin{lem}\label{pressure}
      There exists a sufficiently large constant $C>0$ such that
      $$\Pcal\leq C\left(\Pcal_{1}+\Pcal_{2}\norm{{\bf u}}_{Y_{\tau,\ell_x,\ell_y}}\right) $$
      where
      \begin{equation*}
      \begin{aligned}
      \Pcal_{1} &=\abs{\bf u}_{1,\ell_x,\ell_y,\infty}\abs{\bf u}_{3,\ell_x,\ell_y}+\abs{\bf u}_{2,\ell_x,\ell_y,\infty}\abs{\bf u}_{2,\ell_x,\ell_y}+\abs{\bf u}_{1,\ell_x,\ell_y,\infty}\abs{\bf u}_{2,\ell_x,\ell_y}\\
      &\quad+\tau\left(\abs{\bf u}_{2,\ell_x,\ell_y,\infty}\abs{\bf u}_{3,\ell_x,\ell_y}+\abs{\bf u}_{1,\ell_x,\ell_y,\infty}\abs{\bf u}_{3,\ell_x,\ell_y}+\abs{\bf u}_{2,\ell_x,\ell_y,\infty}\abs{u}_{2,\ell_x,\ell_y}\right) \\
                    &\quad+\tau^2\abs{\bf u}_{2,\ell_x,\ell_y,\infty}\abs{\bf u}_{3,\ell_x,\ell_y}+\tau^3\abs{\bf u}_{3,\ell_x,\ell_y,\infty}\abs{\bf u}_{3,\ell_x,\ell_y}
      \end{aligned}
      \end{equation*}
      and
     \begin{equation*}
     \begin{aligned}
     \Pcal_{2} &=\tau\abs{\bf u}_{1,\ell_x,\ell_y,\infty}+\tau^2\left(\abs{\bf u}_{2,\ell_x,\ell_y,\infty}+\abs{\bf u}_{1,\ell_x,\ell_y,\infty}\right)+\tau^3\left(\abs{\bf u}_{3,\ell_x,\ell_y,\infty} \right.\\
     &\left.\quad+\abs{\bf u}_{2,\ell_x,\ell_y,\infty}\right)
     +\left(\tau^2+\tau^{5/2}+\tau^3 \right)\norm{\bf u}_{X_{\tau,\ell_x,\ell_y}}+\tau^4\abs{\bf u}_{3,\ell_x,\ell_y,\infty}
     \end{aligned}
     \end{equation*}
    \end{lem}

Let $m\geq6$ be fixed.  With Sobolev embedding theorem and the lemmas above and (\ref{3.2}), we have
    \begin{equation}\label{3.3}
     \begin{aligned}
     \frac{d}{dt}\norm{{\bf u}(t,\cdot)}_{X_{\tau(t),\ell_x,\tau_y}} &\leq {\tau}^\prime(t) \norm{{\bf u}(t,\cdot)}_{Y_{\tau(t),\ell_x,\ell_y}}+C(1+\tau(t)^3) \norm{{\bf u}(t,\cdot)}_{H^m_{\ell_x,\ell_y}}^2\\
                                                      &\quad+C \tau(t)\left(\norm{\bf u}_{L^\infty} +\abs{\bf u}_{1,\ell_x,\ell_y,\infty}\right) \norm{\bf u}_{Y_{\tau,\ell_x,\ell_y}}  \\
                                                     &\quad+ C (\tau(t)^2+\tau(t)^4)\norm{\bf u}_{{ H}^m_{\ell_x,\ell_y}}\norm{\bf u}_{Y_{\tau,\ell_x,\ell_y}} \\
                                                     &\quad+                 C        (\tau(t)^2+\tau(t)^3)\norm{\bf u}_{X_{\tau(t),\ell_x,\ell_y}} \norm{\bf u}_{Y_{\tau,\ell_x,\ell_y}}
     \end{aligned}
    \end{equation}
where the constant $C$ is independent of $u,v$. If $\tau(t)$ decreases fast enough such that
    \begin{equation}\label{3.4}
    \begin{aligned}
     {\tau}^\prime(t)+C\tau(t)\left(\norm{\bf u}_{L^\infty}+\abs{\bf u}_{1,\ell_x,\ell_y,\infty}\right)&+C(\tau(t)^2+\tau(t)^4)\norm{\bf u}_{ H^m_{\ell_x,\ell_y}}\\
     &+C(\tau(t)^2+\tau(t)^3)\norm{\bf u}_{X_{\tau,\ell_x,\ell_y}}\leq 0
    \end{aligned}
    \end{equation}
Then (\ref{3.3}) implies
    \begin{equation*}
     \frac{d}{ dt}\norm{{\bf u}(t)}_{X_{\tau(t),\ell_x,\ell_y}}\leq C(1+\tau(0)^3)\norm{\bf u}^2_{H^m_{\ell_x,\ell_y}}
    \end{equation*}
Therefore
    \begin{equation}\label{priori}
    \begin{aligned}
     \norm{{\bf u}(t)}_{X_{\tau(t),\ell_x,\ell_y}}\leq \norm{{\bf u}_0}_{X_{\tau_0,\ell_x,\ell_y}}+&C_{\tau(0)}\int_0^t \norm{{\bf u}(s,\cdot)}^2_{ H^m_{\ell_x,\ell_y}}ds
     \end{aligned}
    \end{equation}
for all $0<t<\infty$, where $C_{\tau(0)}=C\left(1+\tau(0)^3\right)$. As $\tau(t)$ is chosen to be a decrease function, a sufficient condition for (\ref{3.4}) to hold is that
    \begin{equation}\label{3.5}
    \begin{aligned}
     \tau^\prime(t) &+C\left(\norm{\bf u}_{L^\infty}+\abs{{\bf u}(t)}_{1,\ell_x,\ell_y,\infty}\right)\tau(t)\\
     &+C\tau(t)^2\left[C_{\tau(0)}^\prime \norm{{\bf u}(t)}_{{ H}^m_{\ell_x,\ell_y}}
     +C_{\tau(0)}^{\prime\prime} M(t) \right]\leq 0,
    \end{aligned}
    \end{equation}
where $C_{\tau(0)}^\prime=\left(1+\tau(0)^2\right), C_{\tau(0)}^{\prime\prime}=1+\tau(0)$. Set
\[
M(t)=\norm{{\bf u}_0}_{X_{\tau_x,\ell_y}}+C_{\tau(0)}\int_0^t \norm{{\bf u}(s)}^2_{H^m_{\ell_x,\ell_y}}ds.
\]
and denote
    \begin{equation*}
     G(t)=\exp\left[C\int_0^t \left(\norm{\bf u(s)}_{L^\infty}+\abs{{\bf u}(s)}_{1,\ell_x,\ell_y,\infty}\right) ds \right].
    \end{equation*}
By Proposition \ref{prop1}  we can choose the constant $C>0$ is taken largely enough such that
\[\norm{{\bf u}(t)}^2_{{ H}^m_{\ell_x,\ell_y}}\leq\norm{{\bf u}_0}^2_{{ H}^m_{\ell_x,\ell_y}} G(t).\]
It then follows that (\ref{3.5}) is satisfied if we let
    \begin{equation*}
     \tau(t)=G(t)^{-1}\frac{1}{\frac{1}{\tau(0)}+C\int_0^t \left[C_{\tau(0)}^\prime \norm{{\bf u}(s)}_{{ H}^m_{\ell_x,\ell_y}}+C_{\tau(0)}^{\prime\prime}M(s)\right]G(s)^{-1}  ds }.
    \end{equation*}
With this decreasing function $\tau(t)$, we can conclude the a priori estimates that are used to prove of Theorem \ref{thsec}.

\section{the commutator estimate}\label{sec4}
In this section we will prove Lemma \ref{commutator}, the method here is similar with \cite{KV2} except for the parts involving  the weight function.
\begin{proof}[Proof of Lemma \ref{commutator}]
We first write the sum as
    \begin{equation*}
      \Ccal=\sum_{m=3}^\infty\sum_{j=1}^m \Ccal_{m,j},
    \end{equation*}
where we denote
    \begin{equation*}
    \begin{aligned}
      \Ccal_{m,j} &=\frac{\tau^{m-3}}{(m-3)!^s}\sum_{\abs{\alpha}=m}\sum_{\abs{\beta}=j,\beta\leq\alpha}{\alpha\choose\beta}
      \bigg(\norm{\comii{y}^{\ell_y}\dd^\beta {\bf u}\cdot\nabla\dd^{\alpha-\beta}u}_{L^2} \\
      &\quad+\norm{\comii{x}^{\ell_x}\dd^\beta {\bf u}\cdot\nabla\dd^{\alpha-\beta}v}_{L^2} \bigg).
    \end{aligned}
    \end{equation*}
Then we split the right side of the above inequality into seven terms according to the values of m and j, and prove the following estimates.

For small  $j$, we have
    \begin{equation*}
     \sum_{m=3}^\infty \Ccal_{m,1}\leq  C\abs{\bf u}_{1,\infty}\abs{\bf u}_{3,\ell_x,\ell_y}+C\tau\abs{\bf u}_{1,\ell_x,\ell_y,\infty} \norm{\bf u}_{Y_{\tau,\ell_x,\ell_y}}
    \end{equation*}
and
    \begin{equation*}
      \sum_{m=3}^\infty \Ccal_{m,2}\leq C\abs{\bf u}_{2,\infty}\abs{\bf u}_{2,\ell_x,\ell_y}+C\tau\abs{ \bf u}_{2,\infty}\abs{\bf u}_{3,\ell_x,\ell_y}+C\tau^2\abs{\bf u}_{2,\infty}\norm{\bf u}_{Y_{\tau,\ell_x,\ell_y}}.
    \end{equation*}
For intermediate $j$, we have
    \begin{equation*}
     \sum_{m=6}^\infty\sum_{j=3}^{[{m\over2}]}\Ccal_{m,j}\leq C\tau^2 \norm{\bf u}_{X_\tau}\norm{\bf u}_{Y_{\tau,\ell_x,\ell_y}}
    \end{equation*}
and
    \begin{equation}\label{4.1}
    \begin{aligned}
     \sum_{m=7}^\infty \sum_{j=[m/2]+1}^{m-3}\Ccal_{m,j} &\leq C(\tau^2+\tau^{5/2}+\tau^3) \norm{\bf u}_{X_\tau}\norm{\bf u}_{Y_{\tau,\ell_x,\ell_y}}
    \end{aligned}.
    \end{equation}
For higher j, we have
    \begin{equation*}
     \sum_{m=5}^\infty \Ccal_{m,m-2}\leq C\tau^2\abs{\bf u}_{3}\abs{\bf u }_{3,\ell_x,\ell_y,\infty}+C\tau^3\abs{\bf u}_{3,\ell_x,\ell_y,\infty}\norm{\bf u}_{Y_\tau},
    \end{equation*}
    \begin{equation*}
     \sum_{m=4}^\infty \Ccal_{m,m-1}\leq C\tau\abs{\bf u}_{3}\abs{\bf u}_{2,\ell_x,\ell_y,\infty}+C\tau^2\abs{\bf u}_{2,\ell_x,\ell_y,\infty}\norm{\bf u}_{Y_\tau}
    \end{equation*}
and
    \begin{equation*}
     \sum_{m=3}^\infty \Ccal_{m,m}\leq C\abs{\bf u}_{1,\ell_x,\ell_y,\infty}\abs{\bf u}_3+C\tau\abs{\bf u}_{1,\ell_x,\ell_y,\infty}\norm{\bf  u}_{Y_\tau}.
    \end{equation*}
    The proof of the above estimates is similar as in \cite{KV2} and we just point out the difference due to the weight function.  The main difference may be caused by the weight function is the estimation of (\ref{4.1}).
Note that for $[m/2]+1\leq j\leq m-3$, with H\"older inequality and [Proposition 3.8, Chapter 13,Section 3,\cite{MT}], one have
    \begin{equation}\label{4.2}
    \begin{aligned}
     \norm{\comii{y}^{\ell_y}\dd^\beta{\bf u}\cdot\nabla\dd^{\alpha-\beta}u}_{L^2} &\leq C\norm{\dd^\beta{\bf u}}_{L^2}\norm{\comii{y}^{\ell_y}\nabla\dd^{\alpha-\beta}u}_{L^2}^{1/2}\\
     &\quad\times\norm{D^2\left(\comii{y}^{\ell_y}\nabla\dd^{\alpha-\beta} u\right)}_{L^2}^{1/2},
    \end{aligned}
    \end{equation}
where we used the notation
    \[
     D^k u=\{\dd^\alpha u: \abs{\alpha}=k\},\quad\ \ \norm{D^k u}_{L^2}=\sum_{\abs{\alpha}=k}\norm{\dd^\alpha u}_{L^2}.
    \]
Note that by Leibniz formula
    \begin{equation}\label{4.3}
    \begin{aligned}
     \norm{D^2\left(\comii{y}^{\ell_y}\nabla\dd^{\alpha-\beta}u \right)}_{L^2} &\leq C\bigg(\norm{\comii{y}^{\ell_y}\nabla\dd^{\alpha-\beta}u}_{L^2}
     +\norm{\comii{y}^{\ell_y}D^1\nabla\dd^{\alpha-\beta}u}_{L^2}\\
     &\quad+\norm{\comii{y}^{\ell_y}D^2\nabla\dd^{\alpha-\beta}u}_{L^2} \bigg)
    \end{aligned}
    \end{equation}
where $C$ is a constant. And here we used the fact that,  observing $0\leq\ell_y\leq1$,
 \[
 \left|\dd_y^2\comii{y}^{\ell_y}\right|\leq C,\quad\ \ \left|\dd_y\comii{y}^{\ell_y}\right|\leq1, \quad 1\leq\comii{y}^{\ell_y},
 \]
for some constant $C$. And similar arguments also applied to $\norm{\comii{x}^{\ell_x}\dd^\beta{\bf u}\cdot\nabla\dd^{\alpha-\beta}v}_{L^2}$.

With (\ref{4.2}) and (\ref{4.3}) , we have
    \begin{equation*}
    \begin{aligned}
      \sum_{m=7}^\infty\sum_{j=[m/2]+1}^{m-3}\Ccal_{m,j}
      &\leq C\sum_{m=7}^\infty \sum_{j=[m/2]+1}^{m-3}\abs{\bf u}_j \abs{\bf u}_{m-j+1,\ell_x,\ell_y}^{1/2}\bigg(\abs{\bf u}_{m-j+1,\ell_x,\ell_y}^{1/2}\\
      &\quad+\abs{\bf u}_{m-j+2,\ell_x,\ell_y}^{1/2}+\abs{\bf u}_{m-j+3,\ell_x,\ell_y}^{1/2} \bigg){m\choose j}\frac{\tau^{m-3}}{(m-3)!^s}
    \end{aligned}
    \end{equation*}
And the estimation of the right side of the above inequality is similar as in \cite{KV2}. So  we omit the details here. The proof of Lemma 4.3 is complete.
\end{proof}

\section{the pressure estimate}\label{sec5}
It can be deduced from the Euler system (\ref{1.1})-(\ref{1.4}) that the pressure term $p$ satisfies
\begin{equation}\label{5.1}
\Delta p=h\quad\text{in}\ \RR^2_+,
\end{equation}
where $h=2(\dd_x u)\dd_y v-2(\dd_y u)\dd_x v$. Taking the values of \eqref{1.2} on $\dd \RR^2_+$ and using \eqref{1.4}, we have the following Neumann boundary condition
\begin{equation}\label{5.2}
\dd_y p|_{y=0}=0 \quad \text{on}\ \dd\RR^2_+ .
\end{equation}
In order to estimate $\norm{\comii{y}^{\ell_y}\dd_x\dd^\alpha p}_{L^2}$ and $\norm{\comii{x}^{\ell_x}\dd_y\dd^\alpha p}_{L^2}$,
we first consider the following Neumann problem, and here we hope to obtain a weighted $H^2$-regularity result.

\begin{lem}\label{lemma3}
Suppose $\phi$ is the smooth solution of the following equation with Neumann boundary condition, and $\psi\in {\rm C}^\infty$
   \begin{equation}\label{5.3}
    \left\{
    \begin{aligned}
      \Delta \phi &=\psi\quad \mbox{in}\quad \RR^2_+,\\
      \dd_y \phi\big|_{y=0} &=0 \quad \mbox{on}\ \dd\RR^2_+.
    \end{aligned}
    \right.
   \end{equation}
Then there exist a  constant $C$ such that for $\forall \alpha\in\NN^2_0$ with $\abs{\alpha}=2$
   \begin{align}
   \norm{\comii{y}^{\ell_y}\dd^\alpha\phi}_{L^2} &\leq C\norm{\comii{y}^{\ell_y}\psi}_{L^2}+C\norm{\dd_x \phi}_{L^2},\label{5.4}\\
   \norm{\comii{x}^{\ell_x}\dd^\alpha\phi}_{L^2} &\leq C\norm{\comii{x}^{\ell_x}\psi}_{L^2}+C\norm{\dd_y \phi}_{L^2},\label{5.5}\\
   \norm{\comii{y}^{\ell_y}\dd_y\dd^\alpha\phi}_{L^2} &\leq C\norm{\comii{y}^{\ell_y}\dd_y\psi}_{L^2}+C\norm{\psi}_{L^2},\label{5.6}\\
   \norm{\comii{x}^{\ell_x}\dd_y\dd^\alpha\phi}_{L^2} &\leq C\norm{\comii{x}^{\ell_x}\dd_y\psi}_{L^2}+C\norm{\psi}_{L^2}.\label{5.7}
   \end{align}
\end{lem}

\begin{proof}
The proof is similar with the classical $H^2$- regularity arguments. Due to the symmetry it suffices to prove (\ref{5.4}) and \eqref{5.6}, since (\ref{5.5}) and (\ref{5.7}) can be proved similarly.

The method is to use integration by parts. We first multiply the first equation of (\ref{5.3}) by $\comii{y}^{2\ell_y}\dd_{xx}\phi$ and integrate over $\RR^2_+$,
    $$
    \norm{\comii{y}^{\ell_y} \dd_{xx}\phi}_{L^2}^2+\int_{\RR^2_+}\comii{y}^{2\ell_y}\dd_{xx}\phi\dd_{yy}\phi dxdy=\left<\comii{y}^{\ell_y}\dd_{xx}\phi ,\comii{y}^{\ell_y}\psi\right>.
    $$
Integrating by parts with the second term, we have
    \begin{equation*}
    \begin{aligned}
    \norm{\comii{y}^{\ell_y} \dd_{xx}\phi}_{L^2}^2+\norm{\comii{y}^{\ell_y} \dd_{xy}\phi}_{L^2}^2 &=\left<\comii{y}^{\ell_y}\dd_{xx}\phi ,\comii{y}^{\ell_y}\psi\right>\\
    &\quad-2\left<(\dd_y\comii{y}^{\ell_y})\dd_x\phi,\comii{y}^{\ell_y}\dd_{xy}\phi\right>.
    \end{aligned}
    \end{equation*}
Using Cauchy-Schwarz inequality and noticing that $\left|\dd_y\comii{y}^{\ell_y}\right|\leq1$ for $0\leq \ell_y\leq1$, we can obtain, for some $0<\eps,\eps'<1$,
    \begin{equation*}
    \begin{aligned}
    \norm{\comii{y}^{\ell_y} \dd_{xx}\phi}_{L^2}^2+\norm{\comii{y}^{\ell_y} \dd_{xy}\phi}_{L^2}^2 &\leq C_\epsilon\norm{\comii{y}^{\ell_y} \psi}_{L^2}^2+\epsilon\norm{\comii{y}^{\ell_y}\dd_{xx}\phi}_{L^2}^2\\
    &\quad+\epsilon^\prime \norm{\comii{y}^{\ell_y}\dd_{xy}\phi}_{L^2}^2+C_{\epsilon^\prime}\norm{\dd_{x}\phi}_{L^2}^2,
    \end{aligned}
    \end{equation*}
and thus
    $$
     \norm{\comii{y}^{\ell_y} \dd_{xx}\phi}_{L^2}+\norm{\comii{y}^{\ell_y}\dd_{xy}\phi}_{L^2}\leq C\norm{\comii{y}^{\ell_y} \psi}_{L^2}+C\norm{\dd_{x}\phi}_{L^2}
    $$
for some constant $C>0$.  Now if we multiply $\comii{y}^{2\ell_y}\dd_{yy}\phi$ on both sides of \eqref{5.3} and do the procedure as above, we can obtain
    $$
    \norm{\comii{y}^{\ell_y} \dd_{yy}\phi}_{L^2}+\norm{\comii{y}^{\ell_y}\dd_{xy}\phi}_{L^2}\leq C\norm{\comii{y}^{\ell_y}\psi}_{L^2}+C\norm{\dd_{x}\phi}_{L^2}.
    $$
Then we have proven (\ref{5.4}). To prove (\ref{5.6}), we first apply $\dd_y$ on equation (\ref{5.3}) to get
    \begin{equation}\label{5.8}
    \left\{
    \begin{aligned}
      \Delta \dd_y\phi &=\dd_y \psi\quad \text{in}\ \RR^2_+,\\
      \dd_y\phi\big|_{y=0} &=0 \quad \mbox{on}\ \dd\RR^2_+.
    \end{aligned}
    \right.
   \end{equation}
Denote $\Phi\triangleq\dd_y \phi$, then we have
    \begin{equation}\label{Phi}
    \left\{
    \begin{aligned}
      \Delta \Phi &=\dd_y \psi\quad \text{in}\ \RR^2_+,\\
      \Phi\big|_{y=0} &=0 \quad \mbox{on}\ \dd\RR^2_+.
    \end{aligned}
    \right.
   \end{equation}
And with this Dirichlet boundary problem for $\Phi$, we multiply \eqref{Phi} by $\comii{y}^{2\ell_y}\dd_{xx}\Phi$ and integrate over $\RR^2_+$.
    \begin{equation}\label{Phi1}
    \norm{\comii{y}^{\ell_y}\dd_{xx}\Phi}_{L^2}^2+\int_{\RR^2_+}\dd_{yy}\Phi \comii{y}^{2\ell_y}\dd_{xx}\Phi dxdy =\int_{\RR^2_+}\dd_y \psi \comii{y}^{2\ell_y}\dd_{xx}\Phi dxdy
    \end{equation}
Since $\Phi$ vanish at infinity and $\Phi\big|_{y=0}$, then $\dd_y\Phi \comii{y}^{2\ell_y}\dd_{xx}\Phi\bigg|_{y=0}=0$. We can integrate by parts to obtain
    \begin{align*}
     \int_{\RR^2_+}\dd_{yy}\Phi \comii{y}^{2\ell_y}\dd_{xx}\Phi dxdy &=\int_{\RR^2_+} \dd_{xy}\Phi \dd_y\left(\comii{y}^{2\ell_y} \dd_x\Phi \right)dxdy\\
        &=\norm{\comii{y}^{\ell_y}\dd_{xy}\Phi}_{L^2}^2+2\int_{\RR^2_+}\comii{y}^{\ell_y}\dd_{xy}\Phi \bigg( \dd_y\comii{y}^{\ell_y}\dd_x\Phi\bigg)dxdy
    \end{align*}
Substituting the above equality into \eqref{Phi1} and using H\"older inequality on the right hand side, we then obtain
    $$
     \norm{\comii{y}^{\ell_y}\dd_{xx}\Phi}_{L^2}^2+\norm{\comii{y}^{\ell_y}\dd_{xy}\Phi}_{L^2}^2\leq C\norm{\comii{y}^{\ell_y}\dd_{y}\psi}_{L^2}^2+C^\prime\norm{\dd_{x}\Phi}_{L^2}^2
    $$
If we multiply \eqref{Phi} by $\comii{y}^{2\ell_y}\dd_{yy}\Phi$, then we can proceed as above to obtain
    $$
     \norm{\comii{y}^{\ell_y}\dd_{yy}\Phi}_{L^2}^2+\norm{\comii{y}^{\ell_y}\dd_{xy}\Phi}_{L^2}^2\leq C\norm{\comii{y}^{\ell_y}\dd_{y}\psi}_{L^2}^2+C^\prime\norm{\dd_{x}\Phi}_{L^2}^2
    $$
With the use of the classical $H^2$ regularity result, we have
    $$
     \norm{\dd_x\Phi}_{L^2}=\norm{\dd_{xy}\phi}_{L^2}\leq C\norm{\psi}_{L^2}
    $$
Combining the above three equalities, we can prove \eqref{5.6}. \eqref{5.5} and \eqref{5.7} can be proved similarly.
\end{proof}

\begin{remark}
The terms of order one on right side of (\ref{5.4})-(\ref{5.5}) are created by differentiating on the weight functions $\comii{x}^{\ell_x}$ and $\comii{y}^{\ell_y}$ when integrating by parts. And this is the main reason why we need the constants $\ell_x,\ell_y$ to be in the interval $[0,1]$.
\end{remark}

For higher order regularity estimates, we need the following lemma.

\begin{lem}\label{lemma4}
Suppose $g$ is a smooth solution of
   \begin{equation}\label{5.9}
   \left\{
   \begin{aligned}
   \Delta g &=f \quad in\ \RR^2_+,\\
   \dd_y g\big|_{y=0} &=0 \quad \mbox{on}\ \dd\RR^2_+,
   \end{aligned}
   \right.
   \end{equation}
with $f\in C^\infty$.  Then there exist a  universal constant $C>0$ such that the following estiamtes
   \begin{equation}\label{5.10}
    \norm{\comii{y}^{\ell_y}\dd_x\dd^\alpha g}_{L^2}\leq C\sum_{\substack{l\in\NN_0,\abs{\beta}=m-1\\ \beta^\prime-\alpha^\prime=2l+1}}\norm{\comii{y}^{\ell_y}\dd^\beta f}_{L^2}+C\norm{\dd_x^{m-2}f}_{L^2}
   \end{equation}
   and
   \begin{equation}\label{5.11}
    \norm{\comii{x}^{\ell_x}\dd_y\dd^\alpha g}_{L^2}\leq C\sum_{\substack{l\in\NN_0,\abs{\beta}=m-1\\ \beta^\prime-\alpha^\prime=2l}}\norm{\comii{x}^{\ell_x}\dd^\beta f}_{L^2}+C\norm{\dd_x^{m-2}f}_{L^2}
   \end{equation}
hold for any $m\geq3$ and any multi-index $\alpha\in\NN^2_0$ such that $\abs{\alpha}=m$ .
\end{lem}

  In (\ref{5.10}) and (\ref{5.11}) we have summation over the set
    $$
    \left\{\beta\in \NN_0^2:\ \abs{\beta}=m-1, \exists~ l\in\NN_0 \ \text{such that}\  \beta^\prime-\alpha^\prime=2l+1\right\}
    $$
    and
    $$
    \left\{\beta\in \NN_0^2:\ \abs{\beta}=m-1, \exists~ l\in\NN_0 \ \text{such that}\  \beta^\prime-\alpha^\prime=2l\right\}
    $$
and similar conventions are used throughout this section.

\begin{proof}
  First by \eqref{5.9},  we use the following induction equality from \cite{KV2}:
    \begin{equation}\label{5.12}
     \dd_y^{2k+2}g=(-1)^{k+1}\dd_x^{2k+2}g+\sum_{j=0}^k (-1)^{k-j}\dd_x^{2k-2j}\dd_y^{2j}f,
    \end{equation}
and   applying $\dd_y$ on the above equation gives
    \begin{equation}\label{5.13}
      \dd_y^{2k+3}g=(-1)^{k+1}\dd_x^{2k+2}\dd_y g+\sum_{j=0}^k (-1)^{k-j}\dd_x^{2k-2j}\dd_y^{2j+1}f.
    \end{equation}
Then for given $\alpha\in\NN_0^2$ with $\abs{\alpha}=m$, we discuss the situations as the value of $\alpha_2$ varies.

Case 1. If $\alpha_2=0$ then $\comii{y}^{\ell_y}\dd_x\dd^\alpha g=\comii{y}^{\ell_y}\dd_x^{m+1}g$ and $$\comii{x}^{\ell_x}\dd_y\dd^\alpha g=\comii{x}^{\ell_x}\dd_y\dd_x^m g$$ Letting $\phi=\dd_x^{m-1}g$ and applying Lemma 5.1, we obtain
    \begin{equation*}
    \begin{aligned}
    \norm{\comii{y}^{\ell_y}\dd_x\dd^\alpha g}_{L^2}&\leq C\norm{\comii{y}^{\ell_y}\dd_x^{m-1}f}_{L^2}+C\norm{\dd_x^m g}_{L^2}\\
    &\leq C\norm{\comii{y}^{\ell_y}\dd_x^{m-1}f}_{L^2}+C\norm{\dd_x^{m-2}f}_{L^2},
    \end{aligned}
    \end{equation*}
and
    \begin{equation*}
    \begin{aligned}
    \norm{\comii{x}^{\ell_x}\dd_y\dd^\alpha g}_{L^2}&\leq C\norm{\comii{x}^{\ell_x}\dd_x^{m-1}f}_{L^2}+C\norm{\dd_y\dd_x^{m-1} g}_{L^2}\\
    &\leq C\norm{\comii{x}^{\ell_x}\dd_x^{m-1}f}_{L^2}+C\norm{\dd_x^{m-2}f}_{L^2}.
    \end{aligned}
    \end{equation*}
In such case, Lemma \ref{lemma4} is proved.

Case 2. If $\alpha_2=1$ then $\comii{y}^{\ell_y}\dd_x\dd^\alpha g=\comii{y}^{\ell_y}\dd_x^{m}\dd_y g$ and $$\comii{x}^{\ell_x}\dd_y\dd^\alpha g=\comii{x}^{\ell_x}\dd_y^2\dd_x^{m-1} g$$ Letting $\phi=\dd_x^{m-1}g$, we can obtain the same result by Lemma \ref{lemma3} as above.\\

Case 3. If $\alpha_2=2k+2\geq2$, then by the induction (\ref{5.12}) we have
   \begin{equation}\label{5.14}
   \begin{aligned}
     \comii{y}^{\ell_y}\dd_x\dd^\alpha g&=\comii{y}^{\ell_y}\dd_y^{2k+2}\dd_x^{\alpha_1+1}g\\
     &=\comii{y}^{\ell_y}(-1)^{k+1}\dd_x^{2k+2}\dd_x^{\alpha_1+1}g+\comii{y}^{\ell_y}\sum_{j=0}^k (-1)^{k-j}\dd_x^{2k-2j}\dd_y^{2j}\dd_x^{\alpha_1+1}f.
   \end{aligned}\end{equation}
Letting $\phi=\dd_x^{2k}\dd_x^{\alpha_1+1}g$, we apply Lemma \ref{lemma3} to obtain
   \begin{equation}\label{5.15}
   \begin{aligned}
   \norm{\comii{y}^{\ell_y}\dd_x^{2k+2}\dd_x^{\alpha_1+1}g}_{L^2}&\leq C\norm{\comii{y}^{\ell_y}\dd_x^{2k}\dd_x^{\alpha_1+1}f}_{L^2}+C\norm{\dd_x^m g}_{L^2}\\
   &\leq C\norm{\comii{y}^{\ell_y}\dd_x^{2k}\dd_x^{\alpha_1+1}f}_{L^2}+C\norm{\dd_x^{m-2}f}_{L^2}.
   \end{aligned}
   \end{equation}
Substituting (\ref{5.15}) into (\ref{5.14}),  we have
   \begin{equation*}
   \begin{aligned}
   \norm{\comii{y}^{\ell_y}\dd_x\dd^\alpha g}_{L^2} &\leq C\sum_{j=0}^k \norm{\comii{y}^{\ell_y}\dd_x^{2k-2j}\dd_y^{2j}\dd_x^{\alpha_1+1}f}_{L^2}+C\norm{\dd_x^{m-2}f}_{L^2}\\
   &\leq C\sum_{\substack{l\in\NN_0,\abs{\beta}=m-1\\ \beta^\prime-\alpha^\prime=2l+1}}\norm{\comii{y}^{\ell_y}\dd^\beta f}_{L^2}+C\norm{\dd_x^{m-2}f}_{L^2}.
   \end{aligned}
   \end{equation*}
And similarly from the induction (\ref{5.13}) equality
    \begin{equation}\label{5.16}
    \begin{aligned}
      \comii{x}^{\ell_x}\dd_y\dd^\alpha g &=\comii{x}^{\ell_x}\dd_y^{2k+3}\dd_x^{\alpha_1}g\\ &=\comii{x}^{\ell_x}(-1)^{k+1}\dd_x^{2k+2}\dd_y \dd_x^{\alpha_1}g+\comii{x}^{\ell_x}\sum_{j=0}^k (-1)^{k-j}\dd_x^{2k-2j}\dd_y^{2j+1}\dd_x^{\alpha_1}f.
    \end{aligned}\end{equation}
Letting $\phi=\dd_x^{2k}\dd_x^{\alpha_1}g$,  we apply Lemma \ref{lemma3} to get
    \begin{equation}\label{5.17}
   \begin{aligned}
   \norm{\comii{x}^{\ell_x}\dd_x^{2k+2}\dd_y\dd_x^{\alpha_1}g}_{L^2}&\leq C\norm{\comii{x}^{\ell_x}\dd_x^{2k}\dd_y\dd_x^{\alpha_1}f}_{L^2}+C\norm{\dd_y^2\dd_x^{2k}\dd_x^{\alpha_1} g}_{L^2}\\
   &\leq C\norm{\comii{x}^{\ell_x}\dd_x^{2k}\dd_y\dd_x^{\alpha_1}f}_{L^2}+C\norm{\dd_x^{m-2}f}_{L^2}.
   \end{aligned}
   \end{equation}
Substituting (\ref{5.17}) into (\ref{5.16}) yields
   \begin{equation*}
   \begin{aligned}
   \norm{\comii{x}^{\ell_x}\dd_y\dd^\alpha g}_{L^2} &\leq C\sum_{j=0}^k \norm{\comii{x}^{\ell_x}\dd_x^{2k-2j}\dd_y^{2j+1}\dd_x^{\alpha_1}f}_{L^2}+C\norm{\dd_x^{m-2}f}_{L^2}\\
   &\leq C\sum_{\substack{l\in\NN_0,\abs{\beta}=m-1\\ \beta^\prime-\alpha^\prime=2l+1}}\norm{\comii{x}^{\ell_x}\dd^\beta f}_{L^2}+C\norm{\dd_x^{m-2}f}_{L^2}.
   \end{aligned}
   \end{equation*}
Thus in such case the lemma is also proved.

Case 4. If $\alpha_2=2k+3\geq3$, then by the induction  we have
   \begin{equation}\label{5.18}
   \begin{aligned}
     \comii{y}^{\ell_y}\dd_x\dd^\alpha g&=\comii{y}^{\ell_y}\dd_y^{2k+3}\dd_x^{\alpha_1+1}g\\
     &=\comii{y}^{\ell_y}(-1)^{k+1}\dd_x^{2k+2}\dd_y\dd_x^{\alpha_1+1}g\\
     &\quad+\comii{y}^{\ell_y}\sum_{j=0}^k (-1)^{k-j}\dd_x^{2k-2j}\dd_y^{2j+1}\dd_x^{\alpha_1+1}f.
   \end{aligned} \end{equation}
Letting $\phi=\dd_x^{2k}\dd_x^{\alpha_1+1}$, then applying Lemma \ref{lemma3}, we have
   \begin{equation*}
   \begin{aligned}
   \norm{\comii{y}^{\ell_y}\dd_x^{2k+2}\dd_y\dd_x^{\alpha_1+1}g}_{L^2}&\leq C\norm{\comii{y}^{\ell_y}\dd_y\dd_x^{2k}\dd_x^{\alpha_1+1}f}_{L^2}+C\norm{\dd_x\dd_y\dd_x^{2k+\alpha_1+1} g}_{L^2}\\
   &\leq C\norm{\comii{y}^{\ell_y}\dd_y\dd_x^{2k}\dd_x^{\alpha_1+1}f}_{L^2}+C\norm{\dd_x^{m-2}f}_{L^2}.
   \end{aligned}
   \end{equation*}
Thus substituting the above estimate into (\ref{5.18}) yields
   \begin{equation*}
   \begin{aligned}
   \norm{\comii{y}^{\ell_y}\dd_x\dd^\alpha g}_{L^2} &\leq C\sum_{j=0}^k \norm{\comii{y}^{\ell_y}\dd_x^{2k-2j}\dd_y^{2j+1}\dd_x^{\alpha_1+1}f}_{L^2}+C\norm{\dd_x^{m-2}f}_{L^2}\\
   &\leq C\sum_{\substack{l\in\NN_0,\abs{\beta}=m-1\\ \beta^\prime-\alpha^\prime=2l+1}}\norm{\comii{y}^{\ell_y}\dd^\beta f}_{L^2}+C\norm{\dd_x^{m-2}f}_{L^2}.
   \end{aligned}
   \end{equation*}
 On the other hand, observe
   \begin{equation}\label{5.19}
   \begin{aligned}
      \comii{x}^{\ell_x}\dd_y\dd^\alpha g &=\comii{x}^{\ell_x}\dd_y^{2(k+1)+2}\dd^{\alpha_1}g\\
      &=\comii{x}^{\ell_x}(-1)^{k+2}\dd_x^{2k+4} \dd_x^{\alpha_1}g+\comii{x}^{\ell_x}\sum_{j=0}^{k+1} (-1)^{k+1-j}\dd_x^{2k+2-2j}\dd_y^{2j}\dd_x^{\alpha_1}f.
\end{aligned}    \end{equation}
Then, letting $\phi=\dd_x^{2k+2}\dd_x^{\alpha_1}g$ and applying Lemma \ref{lemma3}, we obtain
   \begin{equation*}
   \begin{aligned}
   \norm{\comii{x}^{\ell_x}\dd_x^{2k+4}\dd_x^{\alpha_1}g}_{L^2} &\leq C\norm{\comii{x}^{\ell_x}\dd_x^{2k+2}\dd_x^{\alpha_1}f}_{L^2}+C\norm{\dd_y \dd_x^{2k+2+\alpha_1}g}_{L^2}\\
   &\leq C\norm{\comii{x}^{\ell_x}\dd_x^{2k+2}\dd_x^{\alpha_1}f}_{L^2}+C\norm{\dd_x^{m-2}f}_{L^2},
   \end{aligned}
   \end{equation*}
which along with (\ref{5.19}) yields
   \begin{equation*}
   \begin{aligned}
   \norm{\comii{x}^{\ell_x}\dd_y\dd^\alpha g}_{L^2} &\leq C\sum_{j=0}^{k+1} \norm{\comii{x}^{\ell_x}\dd_x^{2k+2-2j}\dd_y^{2j}\dd_x^{\alpha_1}f}_{L^2}+C\norm{\dd_x^{m-2}f}_{L^2}\\
   &\leq C\sum_{\substack{l\in\NN_0,\abs{\beta}=m-1\\ \beta^\prime-\alpha^\prime=2l+1}}\norm{\comii{x}^{\ell_x}\dd^\beta f}_{L^2}+C\norm{\dd_x^{m-2}f}_{L^2}.
   \end{aligned}
   \end{equation*}
So in this case Lemma \ref{lemma4} is also proved.   Thus for all $\alpha$ such that $\abs{\alpha}=m$ we have proved Lemma \ref{lemma4}.
\end{proof}

Now we come to the proof of Lemma \ref{pressure}.

\begin{proof}[Proof of Lemma \ref{pressure}]
Apply Lemma \ref{lemma4} with equation (\ref{5.1})-(\ref{5.2}) we have
   \begin{equation*}
   \begin{aligned}
    \Pcal &=\sum_{m=3}^\infty \frac{\tau^{m-3}}{(m-3)!^s}\sum_{\abs{\alpha}=m}\bigg(\norm{\comii{y}^{\ell_y}\dd_x\dd^\alpha p}_{L^2}+\norm{\comii{x}^{\ell_x}\dd_y\dd^\alpha p}_{L^2} \bigg)\\
    &\leq C\sum_{m=3}^\infty \frac{\tau^{m-3}}{(m-3)!^s}\sum_{\abs{\alpha}=m}\bigg(\sum_{\substack{l\in\NN_0,\abs{\beta}=m-1\\ \beta^\prime-\alpha^\prime=2l+1}}\norm{\comii{y}^{\ell_y}\dd^\beta h}_{L^2}\\
    &\quad+\sum_{\substack{l\in\NN_0,\abs{\beta}=m-1\\ \beta^\prime-\alpha^\prime=2l}}\norm{\comii{x}^{\ell_x}\dd^\beta h}_{L^2}+\norm{\dd_x^{m-2}h}_{L^2} \bigg)\\
    &\leq C\sum_{m=3}^\infty \frac{m\tau^{m-3}}{(m-3)!^s}\bigg(\sum_{\abs{\beta}=m-1}\norm{\comii{y}^{\ell_y}\dd^\beta h}_{L^2}\\
    &\quad+\sum_{\abs{\beta}=m-1}\norm{\comii{x}^{\ell_x}\dd^\beta h}_{L^2}
    +\norm{\dd_x^{m-2}h}_{L^2} \bigg),
   \end{aligned}
   \end{equation*}
If we exchange the order of the summation, we can obtain, by direct verification,
   \[
   \sum_{\abs{\alpha}=m}\sum_{\substack{l\in\NN_0,\abs{\beta}=m-1\\ \beta^\prime-\alpha^\prime=2l+1}}\norm{\comii{y}^{\ell_y}\dd^\beta h}_{L^2}\leq Cm\sum_{\abs{\beta}=m-1}\norm{\comii{y}^{\ell_y}\dd^\beta h}_{L^2}
   \]
   and
   \[
   \sum_{\abs{\alpha}=m}\sum_{\substack{l\in\NN_0,\abs{\beta}=m-1\\ \beta^\prime-\alpha^\prime=2l}}\norm{\comii{x}^{\ell_x}\dd^\beta h}_{L^2}\leq Cm\sum_{\abs{\beta}=m-1}\norm{\comii{x}^{\ell_x}\dd^\beta h}_{L^2}.
   \]
And direct computation also gives
   \[
   \sum_{\abs{\alpha}=m}\norm{\dd_x^{m-2}h}_{L^2}\leq m\norm{\dd_x^{m-2}h}_{L^2}.
   \]
Since $h=2(\dd_x u)\dd_y v-2(\dd_y u)\dd_x v$, then we have, for arbitrary $\beta\in\NN^2_0$
   \[
   \norm{\comii{y}^{\ell_y}\dd^\beta h}_{L^2}\leq C\sum_{\gamma\leq\beta}{\beta\choose\gamma}\norm{\comii{y}^{\ell_y}\left|\dd^\gamma\nabla{ u}\right|\left|\dd^{\beta-\gamma}{\nabla v} \right|}_{L^2},
   \]
   and
   \[
   \norm{\comii{x}^{\ell_x}\dd^\beta h}_{L^2}\leq C\sum_{\gamma\leq\beta}{\beta\choose\gamma}\norm{\comii{x}^{\ell_x}\left|\dd^\gamma\nabla{ u}\right|\left|\dd^{\beta-\gamma}{\nabla v} \right|}_{L^2}.
   \]
So with these inequalities, we have
   \begin{equation*}
   \begin{aligned}
   \Pcal_{w}&=C\sum_{m=3}^{\infty}\frac{m\tau^{m-3}}{(m-3)!^s}\sum_{\abs{\beta}=m-1}
   \bigg(\norm{\comii{y}^{\ell_y}\dd^\beta h}_{L^2}+\norm{\comii{x}^{\ell_x}\dd^\beta h}_{L^2}\bigg)\\
   &\leq C\sum_{m=3}^{\infty}\frac{m\tau^{m-3}}{(m-3)!^s}\sum_{\stackrel{ \abs{\beta}=m-1} {0\leq\gamma\leq\beta} } {\beta\choose\gamma}\bigg(\norm{\comii{y}^{\ell_y}\left|\dd^\gamma\nabla u\right|\left|\dd^{\beta-\gamma}\nabla v\right|}_{L^2}\\
   &\quad+\norm{\comii{x}^{\ell_x}\left|\dd^\gamma\nabla u\right|\left|\dd^{\beta-\gamma}\nabla v\right|}_{L^2}\bigg)
   \end{aligned}
   \end{equation*}
and
   \begin{equation*}
   \begin{aligned}
   \Pcal_x&=C\sum_{m=3}^{\infty}\frac{m\tau^{m-3}}{(m-3)!^s}\norm{\dd_x^{m-2} h}_{L^2}\\
   &\leq C\sum_{m=3}^{\infty}\frac{m\tau^{m-3}}{(m-3)!^s}\sum_{0\leq j\leq m-2}{m-2\choose j}\norm{\left|\dd_x^j\nabla u\right|\left| \dd_x^{m-2-j}\nabla v\right|}_{L^2}.
   \end{aligned}
   \end{equation*}
Then we have
   \[
    \Pcal\leq \Pcal_w+\Pcal_x.
   \]
The rest part is to estimate $\Pcal_w$ and $\Pcal_x$.

We first estimate $\Pcal_w$.  To do so we split the summation into
   \begin{equation}\label{5.20}
   \Pcal_{w}\leq C\sum_{m=3}^\infty\sum_{j=0}^{m-1}\Pcal_{w,m,j},
   \end{equation}
where
   \begin{equation*}
   \begin{aligned}
   \Pcal_{w,m,j}&=\frac{m\tau^{m-3}}{(m-3)!^s}\sum_{\abs{\beta}=m-1}\sum_{\abs{\gamma}=j}{m-1\choose j}\bigg(\norm{\comii{y}^{\ell_y}\left|\dd^\gamma\nabla u\right|\left|\dd^{\beta-\gamma}\nabla v\right|}_{L^2}\\
   &\quad+\norm{\comii{x}^{\ell_x}\left|\dd^\gamma\nabla u\right|\left|\dd^{\beta-\gamma}\nabla v\right|}_{L^2}  \bigg)
   \end{aligned}
   \end{equation*}
Moreover
we split the right side of (\ref{5.20}) into seven terms according to the values of $m$ and $j$. For lower $j$, we have
   \begin{equation*}
    \sum_{m=3}^\infty \Pcal_{w,m,0} \leq C\abs{\bf u}_{1,\ell_x,\ell_y,\infty}\abs{\bf u}_{3,\ell_x,\ell_y}+C\tau\abs{\bf u}_{1,\ell_x,\ell_y,\infty}\norm{\bf u}_{Y_{\tau,\ell_x,\ell_y}},
   \end{equation*}
   \begin{equation*}
   \begin{aligned}
    \sum_{m=3}^\infty \Pcal_{w,m,1} &\leq C\abs{\bf u}_{2,\ell_x,\ell_y,\infty}\abs{\bf  u}_{2,\ell_x,\ell_y}+C\tau\abs{\bf u}_{2,\ell_x,\ell_y,\infty}\abs{\bf  u}_{3,\ell_x,\ell_y}\\
    &\quad+C\tau^2\abs{\bf u}_{2,\ell_x,\ell_y,\infty}\norm{\bf u}_{Y_{\tau,\ell_x,\ell_y}}
   \end{aligned}
   \end{equation*}
   and
   \begin{equation*}
   \sum_{m=5}^\infty \Pcal_{w,m,2} \leq C\tau^2\abs{\bf u}_{3,\ell_x,\ell_y,\infty}\abs{\bf  u}_{3,\ell_x,\ell_y}+C\tau^3\abs{\bf u}_{3,\ell_x,\ell_y,\infty}\norm{\bf u}_{Y_{\tau,\ell_x,\ell_y}}.
   \end{equation*}
For intermediate $j$, we have
   \begin{equation*}
    \sum_{m=8}^\infty \sum_{j=3}^{[m/2]-1}\Pcal_{w,m,j} \leq C(\tau^2+\tau^{5/2}+\tau^3)\norm{\bf u}_{X_{\tau,\ell_x,\ell_y}}\norm{\bf u}_{Y_{\tau,\ell_x,\ell_y}},
   \end{equation*}
   and
   \begin{equation*}
   \begin{aligned}
    \sum_{m=6}^\infty \sum_{j=[m/2]}^{m-3}\Pcal_{w,m,j}\leq C(\tau^2+\tau^{5/2}+\tau^3)\norm{\bf u}_{X_{\tau,\ell_x,\ell_y}}\norm{\bf u}_{Y_{\tau,\ell_x,\ell_y}}.
   \end{aligned}
   \end{equation*}
For higher $j$, we have
   \begin{equation*}
     \sum_{m=4}^\infty \Pcal_{w,m,m-2}\leq C\tau \abs{\bf  u}_{2,\ell_x,\ell_y,\infty}\abs{\bf u}_{3,\ell_x,\ell_y}+C\tau^2\abs{\bf  u}_{2,\ell_x,\ell_y,\infty}\norm{\bf u}_{Y_{\tau,\ell_x,\ell_y}}
   \end{equation*}
   and
   \begin{equation*}
    \sum_{m=3}^\infty \Pcal_{w,m,m-1}\leq C\abs{\bf u}_{1,\ell_x,\ell_y,\infty}\abs{\bf u}_{3,\ell_x,\ell_y}+C\tau\abs{\bf u}_{1,\ell_x,\ell_y,\infty}\norm{\bf u}_{Y_{\tau,\ell_x,\ell_y}}.
   \end{equation*}

In these estimations we used the fact that for vector function ${\bf u}=(u,v)$, the norm of $u$ or $v$ can be bounded by the norm of ${\bf u}$, for example
 \[
   \norm{\comii{y}^{\ell_y}\nabla\dd^\gamma u}_{L^\infty}\leq \abs{\bf u}_{\abs{\gamma}+1,\ell_x,\ell_y,\infty}.
 \]
With this consideration the estimations can be proved similarly by the method of \cite{KV2} and the arguments of the commutator estimates, and we omit the details.

To estimate $\Pcal_{x}$, we proceed as above, and write
    \begin{equation*}
     \Pcal_x\leq C\sum_{m=3}^\infty\sum_{j=0}^{m-2} \Pcal_{x,m,j},
    \end{equation*}
where
    \begin{equation*}
    \Pcal_{x,m,j}=\frac{m\tau^{m-3}}{(m-3)!^s}{{m-2}\choose j} \norm{\left|\dd_x^j\nabla u\right|\left|\dd_x^{m-j-2}\nabla v\right|}_{L^2}.
    \end{equation*}
For lower $j$, we have
    \begin{align*}
    & \sum_{m=3}^\infty \Pcal_{x,m,0}\leq C\abs{\bf u}_{1,\infty}\abs{\bf u}_2+C\tau\abs{\bf u}_{1,\infty}\abs{ \bf u}_3+C\tau^2\abs{\bf u}_{1,\infty}\norm{\bf u}_{Y_\tau},\\
    & \sum_{m=4}^\infty \Pcal_{x,m,1}\leq C\tau\abs{\bf u}_{2,\infty}\abs{\bf u}_2+C\tau^2\abs{\bf u}_{2,\infty}\abs{\bf  u}_3+C\tau^3\abs{\bf u}_{2,\infty}\norm{\bf u}_{Y_\tau},\\
    & \sum_{m=6}^\infty \Pcal_{x,m,2}\leq C\tau^3\abs{\bf u}_{3,\infty}\abs{\bf u}_3+C\tau^4\abs{\bf u}_{3,\infty}\norm{\bf u}_{Y_\tau}.
    \end{align*}
For mediate $j$, we have
    \begin{align*}
    & \sum_{m=8}^\infty\sum_{j=3}^{[m/2]-1} \Pcal_{x,m,j}\leq C\tau^3\norm{\bf u}_{X_\tau}\norm{\bf u}_{Y_\tau}\\
    & \sum_{m=6}^\infty\sum_{j=[m/2]}^{m-3} \Pcal_{x,m,j}\leq C\tau^3\norm{\bf u}_{X_\tau}\norm{\bf u}_{Y_\tau}.
    \end{align*}
Finally for higher $j$, we have
    \begin{align*}
       \sum_{m=5}^\infty \Pcal_{x,m,m-2}\leq C\tau^2\abs{\bf  u}_{1,\infty}\norm{\bf u}_{Y_\tau}.
    \end{align*}
These estimations can be proved similarly as \cite{KV2} with the fact that $\norm{\bf u}_{X_\tau}\leq\norm{\bf u}_{X_{\tau,\ell_x,\ell_y}}$, $\norm{\bf u}_{Y_\tau}\leq\norm{\bf u}_{Y_{\tau,\ell_x,\ell_y}}$ and $\norm{\bf u}_{H^m}\leq\norm{\bf u}_{H^m_{\ell_x,\ell_y}}$. With all these estimations, we can complete the proof Lemma \ref{pressure}.
\end{proof}

{\bf Acknowledgments.}  W. Li would like to appreciate the support from NSF of China (No. 11422106), and  C.-J. Xu was partially supported by ``the Fundamental Research Funds
for the Central Universities'' and the NSF of China (No. 11171261).


\vspace{1cm}

\begin{thebibliography}{1}





\bibitem{cdgg}

 {\sc J.-Y. Chemin,  B. Desjardins,  I. Gallagher and E. Grenier.}
Mathematical geophysics.
An introduction to rotating fluids and the Navier-Stokes equations. {\em Oxford Lecture Series in Mathematics and its Applications, 32. The Clarendon Press,  Oxford University Press, Oxford.2006}

\bibitem{BoB}{\sc J.P.Bourguignon, H.Brezis}. {\em Remarks on the Euler equation}, J. Functional Analysis {\bf 15}(1974), 341-363.




\bibitem{clx1}{\sc H. Chen, W.-X. Li and C.-J. Xu }, {\em
Gevrey regularity of subelliptic Monge-Amp\'{e}re equations in the plane }
Advances in Mathematics {\bf 228}(2011) 1816-1841

 \bibitem{clx2}{\sc
 H. Chen, W.-X. Li and C.-J. Xu},
{\em Gevrey hypoellipticity for a class of kinetic equations }.
Communications in Partial Differential Equations {\bf 36} (2011) 693-728.

 \bibitem{clx3}{\sc
 H. Chen, W.-X. Li and C.-J. Xu},
{\em Analytic smoothness effect of solutions for spatially homogeneous Landau equation },
Journal of Differential Equations {\bf 248} (2010) 77-94.

 \bibitem{clx4}{\sc
 H. Chen, W.-X. Li and C.-J. Xu}
{\em Gevrey hypoellipticity for linear and non-linear Fokker-Planck equations},
Journal of Differential Equations {\bf 246} (2009), 320- 339.

\bibitem{EM}{\sc D.G.Ebin, J.E.Marsden}.\ {\em Groups of diffeomorphisms and the solution of the classical Euler equations for a perfect fluid}, Bull.Amer.Math.Soc.{\bf 75}(1969),962-967.

\bibitem{KV2}{\sc I.Kukavica, V.Vicol}.\  {\em The domain of analyticity of solutions to three-dimensional euler equations in a half space}, Discrete and Continuous Dynamical Systems, Volume {\bf 29}, Number 1(2011), 285-303.

\bibitem{KV3}{\sc I.Kukavica and V.Vicol}.\ {\em On the analyticity and Gevrey class regularity up to the boundary
for the Euler equation},  Nonlinearity. Volume {\bf 24}, Number 3 (2011), 765-796.

\bibitem{Ka}{\sc T.Kato}.\ {\em Nonstationary flows of viscous and idear fluids in $\RR^3$}, J. Functional Analysis {\bf 9}(1972),296-305.

\bibitem{Ka2}{\sc T.Kato}.\ {\em on classical solutions of two dimensional nonstationary Euler equations}, Arch. Rat. Mech. Anal. Vol {\bf 25}(1967), 188-220.

\bibitem{MT}{\sc M.Taylor}.\ {\em Partial Differential Equations III. Nonlinear Equations}, Springer-Verlag, New York, 1996.


\bibitem{T}{\sc R.Temam}.\ {\em On the Euler equations of incompressible perfect fluids}, J. Functional Analysis {\bf 20}(1975),no.1, 32-43.





\bibitem{Foias}{\sc C.Foias, U.Frisch, R.Temam}.\ {\em Existence de solutions $C^\infty$ des \'{e}quations ${d^\prime}$Euler}, C.R.Acad.Sci.Paris.S\'{e}r.A-B {\bf 280} (1975),A505-A508.

\bibitem{Y}{\sc V.I.Yudovich}.\ {\em Non stationary flow of an ideal incompressible liquid}, Zh. Vych. Mat. {\bf 3}(1963), 1032-1066.



\end{thebibliography}
\end{document}